\documentclass[a4paper,11pt]{amsart}
\usepackage{algorithm, algpseudocode, amssymb, amsthm, bm, comment, enumerate, hyperref, longtable, mathtools, tikz, tikz-3dplot, tikz-cd, verbatim, xcolor}
\usepackage[paper=a4paper,                     
            left=27mm,                         
            right=27mm,                        
            top=25mm,                          
            bottom=25mm,                       
            bindingoffset=0mm                  
           ]{geometry}

\theoremstyle{plain}
\newtheorem{theorem}{Theorem}[section]
\newtheorem{conjecture}[theorem]{Conjecture}
\newtheorem{corollary}[theorem]{Corollary}

\newtheorem{proposition}[theorem]{Proposition}

\theoremstyle{definition}
\newtheorem{definition}[theorem]{Definition}

\theoremstyle{remark}
\newtheorem{example}[theorem]{Example}

\newtheorem{remark}[theorem]{Remark}

\DeclareMathOperator{\ZZ}{\mathbb{Z}}
\DeclareMathOperator{\QQ}{\mathbb{Q}}

\DeclareMathOperator{\KK}{\mathbb{K}}
\DeclareMathOperator{\Affine}{\mathbb{A}}
\DeclareMathOperator{\PP}{\mathbb{P}}
\DeclareMathOperator{\TT}{\mathbb{T}}
\DeclareMathOperator{\GG}{\mathbb{G}}

\DeclareMathOperator{\OEIShypergraphs}{A055621}
\DeclareMathOperator{\Bl}{Bl}
\DeclareMathOperator{\Cone}{Cone}
\DeclareMathOperator{\GL}{GL}
\DeclareMathOperator{\Hom}{Hom}
\DeclareMathOperator{\RelInt}{RelInt}

\makeatletter
\def\@tocline#1#2#3#4#5#6#7{\relax
  \ifnum #1>\c@tocdepth
  \else
    \par \addpenalty\@secpenalty\addvspace{#2}
    \begingroup \hyphenpenalty\@M
    \@ifempty{#4}{
      \@tempdima\csname r@tocindent\number#1\endcsname\relax
    }{
      \@tempdima#4\relax
    }
    \parindent\z@ \leftskip#3\relax \advance\leftskip\@tempdima\relax
    \rightskip\@pnumwidth plus4em \parfillskip-\@pnumwidth
    #5\leavevmode\hskip-\@tempdima
      \ifcase #1
       \or\or \hskip 1em \or \hskip 2em \else \hskip 3em \fi
      #6\nobreak\relax
    \dotfill\hbox to\@pnumwidth{\@tocpagenum{#7}}\par
    \nobreak
    \endgroup
  \fi}
\makeatother
\tdplotsetmaincoords{60}{130}
\tdplotsetrotatedcoords{80}{0}{0}

\title[Classifying additive smooth Fano toric varieties]{Classifying additive smooth Fano toric varieties}
\author{Fabi\'an Levic\'an-Santib\'a\~nez}
\address[F. Levic\'an-Santib\'a\~nez]{Fakultät für Mathematik, Universität Wien, Österreich}
\email{fabian.levican@univie.ac.at}
\author{Pedro Montero}
\address[P. Montero]{Departamento de Matem\'atica, Universidad T\'ecnica
  Fe\-de\-ri\-co San\-ta Ma\-r\'\i a, Valpara\'\i
  so, Chile}  \email{pedro.montero@usm.cl}
\date{\today}

\begin{document}

\begin{abstract}
Let $\KK$ be an algebraically closed field of characteristic zero. An irreducible algebraic variety $X$ over $\KK$ of dimension $n$ is called \emph{additive} if it admits a regular action of the additive group $(\KK^n, +)$ with an open orbit, and \emph{uniquely additive} if this action is unique up to isomorphism.

Huang and the second author have previously determined all additive smooth Fano toric threefolds. Here we determine all additive and uniquely additive smooth Fano toric varieties of dimension up to $6$ by computational means, and give a detailed classification for dimension up to $4$. To this effect, we introduce the \texttt{AdditiveToricVarieties} package for \emph{Macaulay2}, a software system for algebraic geometry and commutative algebra, with methods for working with additive group actions on complete toric varieties. Our work relies on results by Arzhantsev, Dzhunusov and Romaskevich, who relate the existence and uniqueness of such actions to conditions on the Demazure roots of the fans corresponding to the toric varieties. We also prove that every smooth complete toric variety of Picard rank two is additive.
\end{abstract}

\maketitle

\tableofcontents

\section{Introduction}

Let $\KK$ be an algebraically closed field of characteristic zero, let $\GG_a = (\KK, +)$ be its additive group, and let $\GG_a^n = \GG_a \times \cdots \times \GG_a$ ($n$ times) be the corresponding commutative unipotent group. For an irreducible algebraic variety $X$ over $\KK$ of dimension $n$, an \emph{additive action} on $X$ is a regular action $\GG_a^n \times X \to X$ with an open orbit. If $X$ is complete, this is equivalent to $X$ being a completion of affine space $\Affine^n$ that is equivariant with respect to the group of translations by elements of $\Affine^n$.

Additive actions naturally appear as a bridge between affine geometry and the geometry of projective varieties. They provide compactifications of affine space that are equivariant with respect to translations, and hence are related to classical classification problems studied by Hirzebruch \cite{Hir54}. Beyond their geometric interest, varieties with additive actions defined over number fields satisfy the Batyrev--Manin--Peyre principle, as shown by Chambert-Loir and Tschinkel \cite{CLT02}, which gives a precise description of the distribution of rational points of bounded height. From a differential-geometric perspective, they also relate to positivity questions: while the ampleness and nefness of the tangent bundle of smooth projective varieties are well understood through the results of Mori \cite{Mori79} and the Campana--Peternell conjecture (see e.g. \cite{MOSCWW15}), Liu proved in \cite{Liu23} that every variety with an additive action has a big tangent bundle. For further properties and techniques used in the study of additive actions, we kindly refer the reader to the survey by Arzhantsev and Zaitseva \cite{AZ22}.

An additive action on a toric variety $X_{\Sigma}$ is \emph{normalized} if it is normalized by the acting torus $\TT$ (i.e., if $\TT$ is a subgroup of the normalizer of $\GG_a^n$ in the automorphism group of $X_{\Sigma}$). Arzhantsev and Romaskevich \cite{ArzhRoma2017Additive} prove that any two normalized additive actions on $X_{\Sigma}$ are isomorphic, and that, if $X_{\Sigma}$ is complete, then the existence of an additive action on $X_{\Sigma}$ implies the existence of a normalized one.

We say $X_{\Sigma}$ is \emph{additive} if it admits an additive action and \emph{uniquely additive} if any such action is isomorphic to the normalized one.

\emph{Macaulay2} \cite{macaulay2} is a software system that is widely used for research in algebraic geometry and commutative algebra. It contains a variety of existing methods which may be used alongside our package. In particular, it implements polyhedral fans and normal toric varieties.\\
\hfill \\
Our paper is organized in the following way:

In Section \ref{section:preliminaries} we give brief and minimal preliminaries, including previous work due to Arzhantsev, Romaskevich and Dzhunusov \cite{ArzhRoma2017Additive, dzhu2021additive, dzhuV2022uniqueness} about additive actions on complete and projective toric varieties, along with classification results for Gorenstein Fano and smooth Fano toric varieties primarily due to Batyrev, K. Watanabe, M. Watanabe, Sato and Øbro \cite{Baty1981Toric, WataWata1982Classification, Baty1999Classification, Sato2000Toward, obro2007algorithm}.

In Section \ref{section:picard-two} we prove that every smooth complete toric variety $X$ of dimension $\dim(X) \geq 2$ and Picard number $\rho(X) = 2$ is additive, and also study certain projective bundles over Hirzebruch surfaces.

In Section \ref{section:package-description} we describe the methods included in the \emph{Macaulay2} \cite{macaulay2} and give examples of usage.

In Section \ref{section:applications} we apply our methods to classify smooth Fano toric varieties $X$ of dimension $\dim(X) \leq 6$ as (uniquely) additive or not, and match these results with existing classifications for $\dim(X) \leq 4$. This is noteworthy, since existential results for additivity of smooth Fano toric varieties had previously only been obtained for the case in which $\dim(X) \leq 3$ by Huang and the second author \cite{HuanMont2020Fano}, who used more geometric techniques.

Finally, in Section \ref{section:further-directions} we suggest some further directions and prove an enumerative result.

\section{Preliminaries}\label{section:preliminaries}

\subsection{Complete toric varieties} \hfill

We refer the reader to \cite{CoxLittSche2011Toric} for details about the general theory of toric varieties.

Let $\TT \cong (\KK^{\times})^n$ be a torus, let $M = \Hom(\TT, \KK^{\times})$ be the lattice of characters of $\TT$, and let $N = \Hom(M, \ZZ)$ be the dual lattice of $M$. Let $M_{\QQ}$ and $N_{\QQ}$ be the corresponding rational vector spaces, and let $\langle \cdot,\cdot\rangle : N_{\QQ}\times M_{\QQ}\to \QQ$ be the corresponding non-degenerate duality pairing.  Unless specified otherwise, $M \cong \ZZ^n$.

For a fan $\Sigma$ in $N_{\QQ}$, we denote by $\Sigma(1)$ the set of rays and by $X_{\Sigma}$ the corresponding toric variety of dimension $n$. For each ray $\rho \in \Sigma(1)$, we denote by $p_{\rho}$ the unique primitive lattice vector on $\rho$.

From now until the end of the paper, all varieties are assumed to be irreducible, and all fans $\Sigma$ and toric varieties $X_{\Sigma}$ are assumed to be \emph{complete}.

\begin{definition}
    For each ray $\rho \in \Sigma(1)$, define the set
    \[ 
    \mathfrak{R}_\rho \coloneq \{ m \in M: \quad \langle p_{\rho}, m \rangle = -1 \text{ and } \langle p_{\rho'}, m \rangle \geq 0 \text{ for all } \rho' \in \Sigma(1) \setminus \{ \rho \} \}.
    \]
    Define also $\mathfrak{R} \coloneq \bigsqcup_{\rho \in \Sigma(1)} \mathfrak{R}_{\rho}$. The elements of $\mathfrak{R}$ are called the \textbf{Demazure roots} of $\Sigma$.
\end{definition}

\begin{proposition}
    Let $\rho \in \Sigma(1)$. Define the set
    \[
    \mathfrak{R}_{\rho}' \coloneq \{ m \in M_{\QQ}: \quad \langle p_{\rho}, m \rangle = -1 \text{ and } \langle p_{\rho'}, m \rangle \geq 0 \text{ for all } \rho' \in \Sigma(1) \setminus \{ \rho \} \}.
    \]
    Then, $\mathfrak{R}_{\rho}'$ is bounded. In particular, the set $\mathfrak{R}$ of Demazure roots of $\Sigma$ is finite.
\end{proposition}

\begin{proof}
    In \cite[\S 4, Lemme 7]{Dem70}, Demazure proves that $\mathfrak{R}$ is finite under a related hypothesis, namely that the fan $\Sigma$ is \emph{semi-complete}. This is equivalent to $\Sigma$ being \emph{complete} for toric varieties defined over a field. In any case, we adapt his proof to show that $\mathfrak{R}_{\rho}'$ is bounded if $\Sigma$ is complete.

    Assume $\mathfrak{R}_{\rho}'$ is not bounded. Then, since it is a rational polyhedron, it contains an affine ray $x + \QQ_{\geq 0}a$, where $x \in \mathfrak{R}_{\rho}'$ and $a \in N_{\QQ}$. But $\langle p_{\rho}, x + \QQ_{\geq 0}a \rangle = \langle p_{\rho}, x \rangle + \QQ_{\geq 0}\langle p_{\rho} , a \rangle = -1$, so $\langle p_{\rho} , a \rangle = 0$ and $\langle p_{\rho'} , a \rangle \geq 0$ for all $\rho' \in \Sigma(1) \setminus \{\rho\}$. Then $\Sigma$ is not complete.
\end{proof}

\begin{definition}
    Let $S = \{e_1, \ldots, e_n\} \subseteq \mathfrak{R}$. If there exists a linear order $\rho_1, \ldots, \rho_{\# \Sigma(1)}$ on $\Sigma(1)$ such that $\langle p_i, e_j \rangle = -\delta_{ij}$ for all $i, j = 1, \ldots, n$, then $S$ is called a \textbf{complete collection} of Demazure roots of $\Sigma$.
\end{definition}

\begin{theorem}\cite[Theorem 1]{ArzhRoma2017Additive}\label{theorem:complete-collections}
    Complete collections of Demazure roots of $\Sigma$ are in one-to-one correspondence with normalized additive actions on $X_{\Sigma}$.
\end{theorem}

\begin{theorem}\cite[Theorem 2]{ArzhRoma2017Additive}
    Any two normalized actions on $X_{\Sigma}$ are isomorphic.
\end{theorem}

\begin{theorem}\cite[Theorem 3]{ArzhRoma2017Additive}
    The following are equivalent:
    \begin{enumerate}[1.]
        \item There exists an additive action on $X_{\Sigma}$.
        \item There exists a normalized additive action on $X_{\Sigma}$.
    \end{enumerate}
\end{theorem}

The following proposition will reduce the search space of the problem of finding complete collections of Demazure roots, since it involves a notion defined in terms of subsets of $\Sigma(1)$ instead of linear orders.

\begin{definition}
    Let $S = \{s_1, \ldots, s_{\# S}\} \subset N_{\QQ}$ be a finite subset. The \textbf{negative octant} of $S$ is the cone of non-positive linear combinations of elements of $S$
    \[
    \left \{ \sum_{i = 1}^{\# S} a_i s_i: \quad a_i \in \QQ \text{ and } a_i \leq 0\right \} \subseteq N_{\QQ}.
    \]
\end{definition}

\begin{proposition}\cite[Proposition 1, Lemma 2]{dzhu2021additive}\label{proposition:negative-octant-criterion}
    The following are equivalent:
    \begin{enumerate}[1.]
        \item $X_{\Sigma}$ is additive.
        \item There exist rays $\rho_1, \ldots, \rho_n \in \Sigma(1)$ such that $B = \{p_1, \ldots, p_n\}$ is a basis of the lattice $N$ and the remaining rays in $\Sigma(1)$ are contained in the negative octant of $B$. In this case, $-p_i^{\vee} \in \mathfrak{R}_i$ for all $i = 1, \ldots, n$ and $-B^{\vee}~=~\{-p_1^{\vee},~\ldots,~-p_n^{\vee}\} \subset \mathfrak{R}$ is a complete collection of Demazure roots of $\Sigma$. Here, $B^{\vee}$ is the \emph{linear algebraic} dual basis of $B$.
    \end{enumerate}
\end{proposition}

\begin{theorem}\cite{dzhuV2022uniqueness}\label{theorem:uniqueness}
    The following are equivalent:
    \begin{enumerate}[1.]
        \item $X_{\Sigma}$ is uniquely additive.
        \item There exist rays $\rho_1, \ldots, \rho_n \in \Sigma(1)$ such that $B = \{p_1, \ldots, p_n\}$ is a basis of the lattice $N$, the remaining rays in $\Sigma(1)$ are contained in the negative octant of $B$, and $\mathfrak{R}_i = \{-p_i^{\vee}\}$ for all $i = 1, \ldots, n$.
    \end{enumerate}
\end{theorem}

The following proposition will further reduce the aforementioned search space. It is inspired by \cite[Theorem 4]{ArzhRoma2017Additive} (here, Theorem \ref{theorem:inscribed-in-a-rectangle}).

\begin{proposition}
    Assume there exist rays $\rho_1, \ldots, \rho_n \in \Sigma(1)$ such that $B = \{p_1, \ldots, p_n\}$ is a basis of the lattice $N$ and the remaining rays in $\Sigma(1)$ are contained in the negative octant of $B$. Then, $\rho_1, \ldots, \rho_n$ are the extremal rays of a maximal cone $\sigma \in \Sigma$.
\end{proposition}

\begin{proof}
Let $u \in \RelInt(\Cone(p_1, \ldots, p_n)) \cap N$ be a point that is also in the relative interior of some maximal cone $\sigma = \Cone(p_1, \ldots, p_m, p_1', \ldots, p_{M}') \in \Sigma$ with $\QQ_{\geq 0} p_i' \in \Sigma(1) \setminus \{\rho_1, \ldots, \rho_n\}$, which exists since $\Sigma$ is complete. We may write $u = \sum_{i = 1}^n \lambda_i p_i$ and $u = \sum_{i = 1}^m \mu_i p_i + \sum_{j = 1}^M \nu_j p_j'$ with $\lambda_i, \mu_i ,\nu_i > 0$. Since each $p_j'$ is in the negative octant of $B$, the rightmost sum contributes non-positively to the coefficient of each $p_i$, so we have that $m = n$. If $\sigma \neq \Cone(p_1, \ldots, p_n)$, then $\sigma$ contains both $p_1'$ and $-p_1'$, which contradicts its strict convexity. The result follows.
\end{proof}

We may now state the two characterisations that we will implement in the core method of our \textit{Macaulay2} package (see Subsection \ref{subsection:is-additive}).

\begin{corollary}\label{corollary:complete-additive}
    The following are equivalent:
    \begin{enumerate}[1.]
        \item $X_{\Sigma}$ is additive.
        \item There exists a maximal cone $\sigma \in \Sigma$ such that the primitive vectors on the extremal rays of $\sigma_P$ form a basis $B = \{p_1, \ldots, p_n\}$ of the lattice $N$ and the remaining rays in $\Sigma_P(1)$ are contained in the negative octant of $B$.
    \end{enumerate}
\end{corollary}

\begin{corollary}\label{corollary:complete-uniquely-additive}
    The following are equivalent:
    \begin{enumerate}[1.]
        \item $X_{\Sigma}$ is uniquely additive.
        \item There exists a maximal cone $\sigma \in \Sigma$ such that the primitive vectors on the extremal rays of $\sigma_P$ form a basis $B = \{p_1, \ldots, p_n\}$ of the lattice $N$, the remaining rays in $\Sigma_P(1)$ are contained in the negative octant of $B$, and $\mathfrak{R}_i = \{-p_i^{\vee}\}$ for all $i = 1, \ldots, n$.
    \end{enumerate}
\end{corollary}

\subsection{Projective toric varieties}\hfill

We refer the reader to \cite{CoxLittSche2011Toric} for details about the interplay between polytopes and projective toric varieties.

From now until the end of the paper, a polytope $P \subset M_{\QQ}$ is a convex lattice polyhedron that is bounded, $n$-dimensional and very ample. We denote by $\Sigma_P$ its (complete) normal fan in $N_{\QQ}$ and by $X_P$ its corresponding projective toric variety. We also denote by $\mathcal{V}(P)$, $\mathcal{E}(P)$ and $\mathcal{F}(P)$ its set of vertices, edges and facets, respectively.

\begin{definition}\label{definition:inscribed-in-a-rectangle}
A polytope $P \subset M_{\QQ}$ is \textbf{inscribed in a rectangle} if there exists a vertex $v_0 \in \mathcal{V}(P)$ such that:
\begin{enumerate}[1.]
    \item The primitive lattice vectors on the edges $E_i \in \mathcal{E}(P)$ of $P$ starting at $v_0$ form a basis $e_1, \ldots, e_n$ of $M$.
    \item $\langle p_F, e_i \rangle \leq 0$ for all $i = 1, \ldots, n$ and $\rho_F \in \Sigma(1)$ corresponding to a facet $F \in \mathcal{F}(P)$ such that $v_0 \notin F$.
\end{enumerate}
\end{definition}

 \begin{figure}[H]
\centering
\begin{tikzpicture}[tdplot_main_coords,scale=0.5,line join=bevel]
\coordinate (A1) at (-1,2,-2) ;
\coordinate (A2) at (-1,0,0) ;
\coordinate (A3) at (-1,2,1) ;
\coordinate (A4) at (-1,0,1) ;
\coordinate (A5) at (2,-1,-2) ;
\coordinate (A6) at (0,-1,0) ;
\coordinate (A7) at (2,-1,1) ;
\coordinate (A8) at (0,-1,1) ;

\fill  (A1)  circle[radius=4pt] node[right] {$v_0$};

\draw (A8) -- (A6) -- (A2) -- (A4) -- cycle  ;
\draw (A5) -- (A7) -- (A8) -- (A6) -- cycle  ;
\draw (A2) -- (A4) -- (A3) -- (A1) -- cycle  ;
\draw (A5) -- (A6) -- (A2) -- (A1) -- cycle  ;
\draw (A5) -- (A7) -- (A3) -- (A1) -- cycle  ;
\draw (A7) -- (A3) -- (A4) -- (A8) -- cycle  ;
\draw [fill opacity=0.7,fill=black!80!white] (A5) -- (A7) -- (A3) -- (A1) -- cycle  ;
\draw [fill opacity=0.7,fill=black!80!white] (A7) -- (A3) -- (A4) -- (A8) -- cycle  ;
\end{tikzpicture}
\caption{The polytope corresponding to $\mathbb{P}(\mathcal{O}_{\mathbb{P}^1\times \mathbb{P}^1}\oplus \mathcal{O}_{\mathbb{P}^1\times \mathbb{P}^1}(1,1))$.} \end{figure}
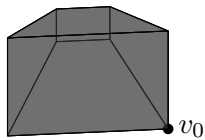

\begin{proposition}
    Let $P \subset M_{\QQ}$ be a polytope. For a vertex $v \in \mathcal{V}(P)$, let $C_v$ be the cone generated by the subset $(P \cap M) - v \subset M_{\QQ}$ and let $\sigma_v = (C_v)^{\vee}$. Then:
    \begin{enumerate}[1.]
        \item For a vertex $v \in \mathcal{V}(P)$, the set of rays parallel to the edges in $\mathcal{E}(P)$ starting at $v$ is equal to the set of extremal rays of $C_v$.
        \item There is a one-to-one correspondence between vertices in $\mathcal{V}(P)$ and maximal cones in $\Sigma$ mapping $v \in \mathcal{V}(P)$ to $\sigma_v \in \Sigma_P$.
    \end{enumerate}
\end{proposition}

\begin{proof}
    \begin{enumerate}[1.]
        \item This follows from the definition of $C_v$.
        \item This follows from \cite[Proposition 2.3.7, Proposition 2.3.8]{CoxLittSche2011Toric}.
    \end{enumerate}
\end{proof}

\begin{theorem}\cite[Theorem 4]{ArzhRoma2017Additive}\label{theorem:inscribed-in-a-rectangle}
Let $P \subset M_{\QQ}$ be a polytope. The following are equivalent:
\begin{enumerate}[1.]
    \item $X_P$ is additive.
    \item $P$ is inscribed in a rectangle.
\end{enumerate}
\end{theorem}

\subsection{Smooth Fano toric varieties}\hfill

\begin{definition}
    Let $X$ be an algebraic variety over $\KK$.
    \begin{enumerate}[1.]
        \item Assume $X$ is normal and complete. $X$ is \textbf{Gorenstein Fano} if the anticanonical divisor $-K_X$ is Cartier and ample (in particular, $X$ is projective).
        \item Assume $X$ is normal and complete. $X$ is \textbf{smooth Fano} if it is Gorenstein Fano and smooth.
    \end{enumerate}
\end{definition}

\begin{definition}
    Let $P$ be a polytope in $M_{\QQ}$ or $N_{\QQ}$.
    \begin{enumerate}[1.]
        \item $P$ is \textbf{smooth} if, for every vertex $v \in \mathcal{V}(P)$, the primitive lattice vectors on the edges $e \in \mathcal{E}(P)$ starting at $v$ form a basis of the lattice.
        \item Assume $0 \in \RelInt(P)$. $P$ is \textbf{reflexive} if its dual $P^{\vee}$ is also a lattice polytope.
        \item Assume $0 \in \RelInt(P)$. $P$ is \textbf{smooth Fano} if the vertices of every facet form a basis of the lattice.
    \end{enumerate}
\end{definition}

\begin{proposition}\cite[Exercise 8.3.6 (a), (b)]{CoxLittSche2011Toric}
    Let $P \subset N_{\QQ}$ be a polytope such that $0_{N_{\QQ}} \in \RelInt(P)$. If $P$ is smooth Fano, then $P$ is reflexive and $P^{\vee}$ is smooth and reflexive. 
\end{proposition}

\begin{theorem}
    Up to isomorphism on both sides:
    \begin{enumerate}[1.]
        \item There is a one-to-one correspondence between pairs $(X_P, D_P)$, where $X_P$ is a smooth projective toric variety and $D_P$ is a torus-invariant Cartier divisor on $X_P$, and smooth polytopes $P \subset M_{\QQ}$.
        \item There is a one-to-one correspondence between Gorenstein Fano toric varieties $X_P$ and reflexive polytopes $P \subset M_{\QQ}$. Here, the torus-invariant Cartier divisor on $X_P$ is $-K_{X_P}$.
        \item There is a one-to-one correspondence between smooth Fano toric varieties $X_{P^{\vee}}$ and smooth Fano polytopes $P \subset N_{\QQ}$. Here, $P^{\vee} \subset M_{\QQ}$ is the dual polytope of $P$ and the torus-invariant Cartier divisor on $X_{P^{\vee}}$ is $-K_{X_{P^{\vee}}}$.
    \end{enumerate}
\end{theorem}

\begin{proof}
    \begin{enumerate}[1.]
        \item This is \cite[Theorem 2.4.3]{CoxLittSche2011Toric}
        \item This is \cite[Theorem 8.3.4]{CoxLittSche2011Toric}
        \item This follows from \cite[Exercise 8.3.6 (b), (c)]{CoxLittSche2011Toric}.
    \end{enumerate}
\end{proof}

\begin{remark}\label{remark:smooth-fano-toric-varieties}
    A potential source of confusion is that, in \cite{CoxLittSche2011Toric}, ``smooth Fano'' polytopes and varieties are simply called ``Fano''. We follow the more common convention by calling them ``smooth Fano''.
    
    Note that smooth Fano polytopes are not necessarily smooth, but their duals are.
    
    Note also that smoothness implies condition 1. in Definition \ref{definition:inscribed-in-a-rectangle} at all vertices $v_0 \in \mathcal{V}(P)$. In particular, for each vertex $v_0 \in \mathcal{V}(P)$, the primitive vectors on the extremal rays of the corresponding maximal cone $\sigma \in \Sigma$ form a basis of the lattice $N$.

    Finally, note that Demazure roots of normal fans $\Sigma_P$ of reflexive polytopes $P \subset M_{\QQ}$ correspond to points in the relative interior of the facets in $\mathcal{F}(P)$.
\end{remark}

It is well-known that there are finitely many isomorphism classes of reflexive polytopes, then also of Gorenstein Fano toric varieties. They were classified algorithmically by Kreuzer and Skarke \cite{Kreuzer2004}.

Smooth Fano toric surfaces are easily classified by inspecting the corresponding reflexive polygons.

Smooth Fano threefolds were classified by Mori and Mukai \cite{mori2003classification}. Smooth Fano toric threefolds were classified by K. Watanabe and M. Watanabe \cite{WataWata1982Classification}, and independently by Batyrev \cite{Baty1981Toric}, who also classified smooth Fano toric fourfolds \cite{Baty1999Classification} along with Sato \cite{Sato2000Toward}.

Higher-dimensional smooth Fano toric varieties were all classified algorithmically by Øbro \cite{obro2007algorithm}. The data up to dimension $6$ is available on the GRDB \cite{gavin2009about}.

\begin{remark}
    The sequence $(a_n)_{n \geq 1}$, where $a_n$ is the number of isomorphism classes of reflexive polytopes, starts with $1, 16, 4319, 473800776, \ldots$

    On the other hand, the sequence $(b_n)_{n \geq 1}$, where $b_n$ is the number of isomorphism classes of smooth Fano polytopes, starts with $1, 5, 18, 124, 866, 7622, \ldots$
    
    Since $a_n$ grows much faster than $b_n$, it is much more feasible to determine existence and uniqueness of additive actions on higher-dimensional smooth Fano toric varieties than it is on Gorenstein Fano toric varieties. 
\end{remark}

\subsection{Betti numbers}\hfill

We now recall and give identities involving invariants of complete toric varieties with \emph{simplicial} fans. Further specializations will allow us to match our results on smooth Fano toric varieties with existing classifications.

In this subsection we assume that $P \subset N_{\QQ}$ is simplicial and that $0_{N_{\QQ}} \in \RelInt(P)$.

\begin{definition}
For $i = 0, \ldots, 2n$, let $H^i(X_{P^{\vee}}, \mathbb{Q})$ be the $i$-th rational cohomology group of $X_{P^{\vee}}$. The $i$-th \textbf{Betti number} of $X_{P^{\vee}}$ is $b_i := \dim(H^{i}(X_{P^{\vee}}, \mathbb{Q}))$ (in particular, $b_2$ is the rank of the Picard group when $X_{P^{\vee}}$ is smooth). For $p = 0, \ldots, n$, define $h_p := b_{2p}$.

The $i$-th \textbf{face number} $f_i$ of $P$ is its number of faces of dimension $i$ (in particular, $f_0 = \# \mathcal{V}(P)$ and $f_{n - 1} = \# \mathcal{F}(P)$). We also set $f_{-1} = 1$.

Finally, we say a Betti number of $X_{P^{\vee}}$ is \textbf{constant} if it is equal to an expression which does not depend on any face number.
\end{definition}

\begin{proposition}\label{prop:poincare-duality}
\begin{enumerate}[1.]
    \item For each $p = 0, \ldots, n$, $h_p = h_{n - p}$.
    \item If $i = 0, \ldots, 2n$ is odd, then $b_i = 0$.
\end{enumerate}
\end{proposition}
\begin{proof}
    \begin{enumerate}[1.]
        \item This follows from Poincaré duality \cite[Section 5]{Fult1993Introduction}.
        \item This is the fact that $H^i(X_{P^{\vee}}, \mathbb{Q})$ is trivial if $i$ is odd \cite[Section 5]{Fult1993Introduction}.
    \end{enumerate}
\end{proof}

\begin{proposition}\label{prop:formula-betti}
For each $p = 0, \ldots, n$,
\begin{equation*}
    h_p = b_{2p} = \sum_{i = p}^d (-1)^{i - p} \binom{i}{p} f_{d - i - 1}.
\end{equation*}
In particular:
\begin{enumerate}[1.]
    \item The constant even Betti numbers of $X_{P^{\vee}}$ are $h_0 = b_0 = \sum_{i = 0}^n (-1)^i f_{n - i - 1} = h_n = b_{2n} = 1$.
    \item $h_1 = b_2 = \sum_{i = 1}^n (-1)^{i - 1} i f_{n - i - 1} = h_{n - 1} = b_{2n - 2} = f_0 - n = \# \mathcal{V}(P) - n$.
\end{enumerate}
\end{proposition}

\begin{proof}
    This follows from \cite[Section 5]{Fult1993Introduction}.
\end{proof}

\section{Toric varieties with Picard number equal to two}\label{section:picard-two}

Let us recall the classification of smooth complete toric varieties $X$ with Picard rank $\rho(X) = 2$ due to Kleinschmidt (compare with \cite[Theorem 7.3.7]{CoxLittSche2011Toric}).

\begin{theorem}\cite[Theorem 1]{Klein1988Classification}\label{theorem:kleinschmidt} Let $X$ be a smooth complete toric variety of dimension $n \geq 2$ with $\rho(X)=2$. Then there are positive integers $r,s\geq 1$ such that $r+s=n$ and that
\[
X\cong \mathbb{P}(\mathcal{O}_{\mathbb{P}^s}\oplus \mathcal{O}_{\mathbb{P}^s}(a_1)\oplus \cdots \oplus \mathcal{O}_{\mathbb{P}^s}(a_r)),
\]
where $0\leq a_1\leq \cdots \leq a_r$.
\end{theorem}

The main result of this section is the following:

\begin{proposition}\label{proposition:picard-two}
Let $X$ be a smooth complete toric variety of dimension $n \geq 2$ with $\rho(X)~=~2$. Then, $X$ is additive. Furthermore, $X$ is uniquely additive if and only if $X \cong \PP^1 \times \PP^1$.
\end{proposition}

\begin{proof}
Let $s,r\geq 1$ and $0\leq a_1\leq \cdots \leq a_r$ be integers such that
\[
X\cong \mathbb{P}(\mathcal{O}_{\mathbb{P}^s}\oplus \mathcal{O}_{\mathbb{P}^s}(a_1)\oplus \cdots \oplus \mathcal{O}_{\mathbb{P}^s}(a_r)),
\]
where $n=r+s$. It follows from \cite[\S 2]{Klein1988Classification} (compare with \cite[Example 7.3.5]{CoxLittSche2011Toric}) that $X$ is the toric variety corresponding to a fan $\Sigma \subseteq N_\mathbb{R}\cong \mathbb{R}^n$ whose rays $\rho_1,\ldots,\rho_{n+2}$ are generated by the following $n+2$ primitive lattice vectors:
\begin{itemize}
\item $u_i:=e_i$ for $i=1,\ldots,r$,
\item $u_{r+1}:=-\sum_{i=1}^r e_i$,
\item $u_{r+j+1}:=e_{r+j}$ for $j=1,\ldots,s$,
\item $u_{n+2}:=\sum_{i=1}^r a_i e_i - \sum_{j=1}^s e_{r+j}$,
\end{itemize}
where $e_i$ is the $i$-th vector of the canonical basis of $\mathbb{R}^n$. Let $e_1^{\vee},\ldots,e_n^{\vee}\in M$ be the basis of $M$ that is dual to $e_1,\ldots,e_n\in N$, and define
\begin{itemize}
\item $m_i:=e_r^{\vee} - e_i^{\vee}$ for $i=1,\ldots,r-1$,
\item $m_{r+1}:=e_r^{\vee}$,
\item $m_{r+1+j}:=-e_{r+j}^{\vee}$ for $j=1,\ldots,s$.
\end{itemize} 
A direct computation shows that the set $\{m_1,\ldots,m_{r-1},m_{r+1},m_{r+2},\ldots,m_{n+1}\}$ is a complete collection of Demazure roots for the fan $\Sigma$ of $X$. More precisely, we have that $m_i\in \mathfrak{R}_{\rho_i}$ for every $i=1,\ldots,r-1,r+1,\ldots,n+1$ and that \footnote{Note that we only require $\langle u_i,m_j\rangle \geq 0$ for $i=r$ and $i=n+2$.}
\[
\langle u_i,m_j\rangle = - \delta_{ij} \textup{ for all }i,j=1,\ldots,r-1,r+1,\ldots,n+1.
\]
It follows from Theorem \ref{theorem:complete-collections} that $X$ is additive.

Furthermore, if there exists $a_i \neq 0$, then $\{e_r^{\vee}, e_r^{\vee} + e_{r + 1}^{\vee}\} \subseteq \mathfrak{R}_{\rho_{r + 1}}$ so, by Theorem \ref{theorem:uniqueness}, $X$ is not uniquely additive. If all $a_i = 0$, then $X \cong \PP^r \times \PP^s$, which is uniquely additive if $r = s = 1$ and admits at least two non-isomorphic additive actions for all other values of $r, s$ by \cite{HassTsch1999Geometry}.
\end{proof}

\begin{remark}
In dimension $3$ or higher, the assumption $\rho(X)=2$ in Proposition \ref{proposition:picard-two} cannot be removed. Indeed, the smooth Fano toric threefold $\operatorname{III}_{25}\cong \mathbb{P}(\mathcal{O}_{\mathbb{P}^1\times \mathbb{P}^1}(1,0)\oplus \mathcal{O}_{\mathbb{P}^1\times \mathbb{P}^1}(0,1))$ is not additive, since $\operatorname{III}_{25}$ is also obtained as the blow-up of $\mathbb{P}^3$ along two disjoint lines and it follows from \cite[Proposition 3.3]{HassTsch1999Geometry} that the invariant subvarieties for all the four additive actions on $\mathbb{P}^3$ are contained in the boundary projective plane.
\end{remark}

The following well-known fact will be useful for computations. It follows in greater generality from Proposition \ref{prop:formula-betti} 2.

\begin{proposition}\label{proposition:picard-number-smooth-complete-toric-variety}
    Let $X_{\Sigma}$ be a smooth complete toric variety. Then, $\rho(X) = \# \Sigma(1) - n$.
\end{proposition}

\subsection{Projective bundles over Hirzebruch surfaces}\hfill

Let $d \geq 0$ be a non-negative integer. Let us recall that the fan of the Hirzebruch surface $\Sigma_d=\mathbb{P}(\mathcal{O}_{\mathbb{P}^1}\oplus \mathcal{O}_{\mathbb{P}^1}(d))$ is given by

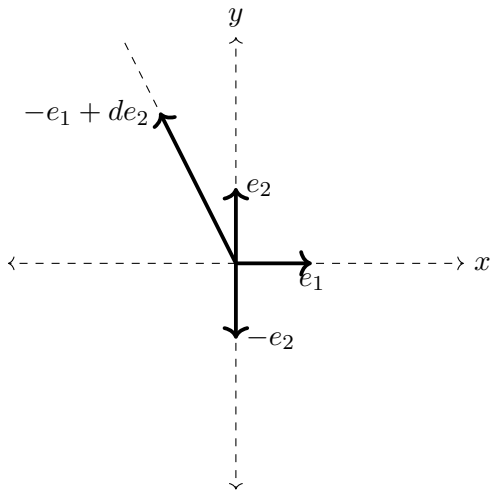
\begin{figure}[ht]
\begin{center}
\begin{tikzpicture}
\draw [dashed,->]  (0,0) -- (3,0);   
\draw [dashed,->]  (0,0) -- (-3,0);   
\node [right] at (3,0) {$x$};
\draw [dashed,->]  (0,0) -- (0,3); 
\draw [dashed,->]  (0,0) -- (0,-3); 
\node [above] at (0,3) {$y$}; 

\draw[->, line width=0.5mm] (0,0) -- (1,0);
\draw[dashed] (0,0) -- (-1.5,3);
\draw[->, line width=0.5mm] (0,0) -- (0,1);
\draw[->, line width=0.5mm] (0,0) -- (0,-1);
\draw[->, line width=0.5mm] (0,0) -- (-1,2);

\node [right] at (0,-1) {$-e_2$};
\node [left] at (-1,2) {$-e_1+de_2$};
\node [below] at (1,0) {$e_1$};
\node [right] at (0,1) {$e_2$};
\end{tikzpicture} 
\end{center}
\caption{The fan of the Hirzebruch surface $\Sigma_d$.}\label{fig:1}
\end{figure}

It follows from \cite[Proposition 7.3.3]{CoxLittSche2011Toric}, that if we consider $D\sim aF+b\xi$ with $a,b\in \mathbb{Z}$, then
\[
\Sigma(d;a,b):=\mathbb{P}(\mathcal{O}_{\Sigma_d}\oplus \mathcal{O}_{\Sigma_d}(aF+b\xi))
\]
is a smooth toric threefold, whose fan is given by

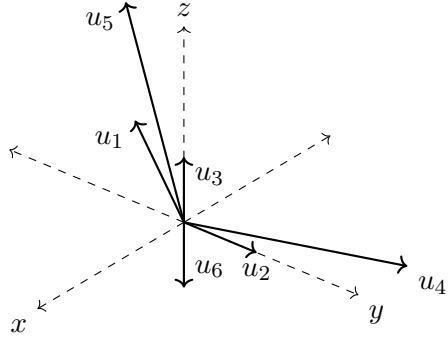
\begin{figure}[ht]
\centering
\begin{tikzpicture}[tdplot_main_coords,scale=0.5,line join=bevel]

\draw[dashed,->] (0,0,0) -- (6,0,0) node[anchor=north east]{$x$};
\draw[dashed,->] (0,0,0) -- (-6,0,0);
\draw[dashed,->] (0,0,0) -- (0,6,0) node[anchor=north west]{$y$};
\draw[dashed,->] (0,0,0) -- (0,-6,0);
\draw[dashed,->] (0,0,0) -- (0,0,6) node[anchor=south]{$z$};

\draw[thick,->] (0,0,0) -- (2,0,4) node[anchor=north east]{$u_1$};
\draw[thick,->] (0,0,0) -- (0,2.5,0) node[anchor=north]{$u_2$};
\draw[thick,->] (0,0,0) -- (0,0,2) node[anchor=north west]{$u_3$};
\draw[thick,->] (0,0,0) -- (-2,6,0) node[anchor=north west]{$u_4$};
\draw[thick,->] (0,0,0) -- (0,-2,6) node[anchor=north east]{$u_5$};
\draw[thick,->] (0,0,0) -- (0,0,-2) node[anchor=south west]{$u_6$};

\end{tikzpicture}
\caption{The fan of $\Sigma(d;a,b)$.} \label{fig:P10}
\end{figure}

Here, the primitive generators of the rays of the fan of $\Sigma(d;a,b)$ are given by
\[
u_1=(1,0,a),\;\; u_2=(0,1,0),\;\; u_3=(0,0,1),\;\; u_4=(-1,d,0),\;\; u_5=(0,-1,b),\;\; u_6=(0,0,-1).
\]
Moreover, by twisting by $\mathcal{O}_{\Sigma_d}(-D)$ if necessary, we may assume that $a\leq b$.

\vspace{2mm}

In order to apply the results in \cite{ArzhRoma2017Additive} to determine whether $\Sigma(d;a,b)$ is additive or not, we will first compute the set of Demazure roots associated to each ray of the fan. More precisely, for $i=1,\ldots,6$, consider
\[
\mathfrak{R}_i:=\{m=(p,q,r)\in \mathbb{Z}^3\;|\;\langle u_i,m\rangle =-1 \textup{ and }\langle u_j,m\rangle \geq 0 \textup{ for all }j\neq i\}.
\]
Then:
\begin{enumerate}[1.]
\item $\mathfrak{R}_1=\{(-1,0,0)\}$.
\item $\mathfrak{R}_2=\{(0,-1,0)\}$.
\item $\mathfrak{R}_3=\{m=(p,q,-1)\in \mathbb{Z}^3\;|\;p\geq a,\;q\geq 0,\;qd\geq p,\;-b\geq q\}$. In particular, $\mathfrak{R}_3\neq \varnothing$ if and only if $b\leq 0$.
\item $\mathfrak{R}_4=\{(1,0,0)\}$.
\item $\mathfrak{R}_5=\{(0,1,0)\}$.
\item $\mathfrak{R}_6=\{m=(p,q,1)\in \mathbb{Z}^3\;|\;p\geq -a,\;q\geq 0,\;qd\geq p,\;b\geq q\}$. In particular, $\mathfrak{R}_6\neq \varnothing$ if and only if $a+bd\geq 0$.
\end{enumerate}

\begin{proposition}\label{proposition:hirzebruch}
Let $d\geq 0$ and $a\leq b$ be integers. Then, the toric threefold
\[
\Sigma(d;a,b):=\mathbb{P}(\mathcal{O}_{\Sigma_d}\oplus \mathcal{O}_{\Sigma_d}(aF+b\xi))
\]
is additive if and only if $b\leq 0$ or $a+bd\geq 0$.
\end{proposition}

\begin{proof}
We first remark that $\mathfrak{R}_3\neq \varnothing$ and $\mathfrak{R}_6 \neq \varnothing$ if and only if $a\leq b \leq 0$ and $0\leq (-b)d\leq a$, which implies that $a=b=0$. On the other hand, $\Sigma(d;0,0)\cong \Sigma_d\times \mathbb{P}^1$ is additive.

Assume first that $\mathfrak{R}_3\neq \varnothing$ and $\mathfrak{R}_6=\varnothing$. In this case, we obtain a complete collection of Demazure roots for the fan of $\Sigma(d;a,b)$ by considering
\[
m_1=(-1,0,0)\in \mathfrak{R}_1,\;m_2=(0,1,0)\in \mathfrak{R}_5,\;m_3=(a,-b,-1)\in \mathfrak{R}_3.
\]
Secondly, we assume that $\mathfrak{R}_6\neq \varnothing$ and $\mathfrak{R}_3=\varnothing$. In this case, we obtain a complete collection of Demazure roots for the fan of $\Sigma(d;a,b)$ by considering
\[
m_1=(-1,0,0)\in \mathfrak{R}_1,\;m_2=(0,1,0)\in \mathfrak{R}_5,\;m_3=(-a,b,1)\in \mathfrak{R}_6.
\]
Therefore, in both cases we conclude from Theorem \ref{theorem:complete-collections} that $\Sigma(d;a,b)$ is additive.

Finally, we note that if $\mathfrak{R}_3$ and $\mathfrak{R}_6$ are empty, then we have at most two linearly independent Demazure roots, and hence in that case $\Sigma(d;a,b)$ is not additive.
\end{proof} 

\section{Package description}\label{section:package-description}

We introduce \texttt{AdditiveToricVarieties}, a new \emph{Macaulay2} \cite{macaulay2} package for working with additive actions on complete toric varieties, with methods based on the results we have discussed so far. We have followed package writing guidelines and standards described in the \emph{Macaulay2} official website. In particular, we have included documentation and examples for each method, which are automatically
generated as HTML files upon installation. We have also implemented automated tests which may be run at any time.

The package also includes a pre-processed database which classifies smooth Fano toric varieties $X$ of dimension at most $6$ as additive (or not) and uniquely additive (or not), as per the database included in the \texttt{NormalToricVarieties} package.

The package dependencies are packages \texttt{NormalToricVarieties} and \texttt{Polyhedra}. The exported methods are \texttt{demazureRoots}, \texttt{isAdditive}, \texttt{listAdditiveSmoothFanoToricVarieties} and \texttt{listUniquelyAdditiveSmoothFanoToricVarieties}.

The package and installation instructions may be found in the first author's \emph{GitHub} repository: \url{https://github.com/iZafiro/AdditiveToricVarieties}.

\subsection{demazureRoots}\hfill

This method takes as input a complete fan $\Sigma$ (an object of type \texttt{Fan}) and optionally a primitive lattice vector $p_i$ on a ray $\rho_i \in \Sigma(1)$ (an object of type \texttt{Matrix}). If only a fan $\Sigma$ is given, it returns a hash table with keys the lattice vectors on all rays in $\Sigma(1)$ (objects of type \texttt{List}) and values the corresponding sets $\mathfrak{R}_i$ (objects of type \texttt{List}). If additionally a vector $p_i$ is given, it only returns the corresponding set $\mathfrak{R}_i$ (an object of type \texttt{List}).

\begin{example}
The projective plane has six Demazure roots:
\begin{verbatim}
i1 : F = fan toricProjectiveSpace 2;
i2 : demazureRootsF = demazureRoots F;
i3 : flatten values demazureRootsF
o3 = {{0, -1}, {1, -1}, {-1, 1}, {-1, 0}, {0, 1}, {1, 0}}
o3 : List
\end{verbatim}

There are two roots corresponding to \texttt{e1}:
\begin{verbatim}
i4 : e1 = {1, 0};
i5 : demazureRootsF#e1
o5 = {{-1, 1}, {-1, 0}}
o5 : List
\end{verbatim}

Alternatively, this can also be computed directly:
\begin{verbatim}
i6 : demazureRoots(F, e1)
o6 = {{-1, 1}, {-1, 0}}
o6 : List
\end{verbatim}
\end{example}

\subsection{additiveActions}\label{subsection:is-additive}\hfill

This method takes as input either a complete fan $\Sigma$ (an object of type \texttt{Fan}), a complete toric variety $X$ (an object of type \texttt{NormalToricVariety}), a polytope $P \subset M_{\QQ}$ (an object of type \texttt{Polyhedron}) or a pair $(n, i)$ with $n = 1, \ldots, 6$ and $i$ an index in the database of smooth Fano toric varieties included in the \texttt{NormalToricVarieties} package.

The method may also take as optional inputs \texttt{getCompleteCollection} (a \texttt{Boolean}, whether the output must include a complete collection of Demazure roots), \texttt{checkIfUniquelyAdditive} (a \texttt{Boolean}, whether the output must include if the variety is uniquely additive), \texttt{fanIsComplete} (a \texttt{Boolean}, whether the fan is known to be complete) and \texttt{fanIsSmooth} (a \texttt{Boolean}, whether the fan is known to be smooth).

The output is a hash table with keys \texttt{"isAdditive"} (a \texttt{Boolean}, whether the variety is additive), \texttt{"completeCollection"} (an object of type \texttt{List}, a complete collection of Demazure roots, if the variety is additive and \texttt{getCompleteCollection} is true) and \texttt{"isUniquelyAdditive"} (a \texttt{Boolean}, whether the variety is uniquely additive, if \texttt{checkIfUniquelyAdditive} is true).

This method is very fast if the input is a pair $(n, i)$, since then the data is fetched from the pre-processed database included in our package. It is also very fast if \texttt{getCompleteCollection} is false, the variety is smooth and the Picard number is equal to two, due to Proposition \ref{proposition:picard-two} and Proposition \ref{proposition:picard-number-smooth-complete-toric-variety}. Furthermore, setting \texttt{getCompleteCollection} or \texttt{checkIfUniquelyAdditive} to false, or \texttt{fanIsComplete} or \texttt{fanIsSmooth} to true or false (the former value more than the latter) may dramatically increase performance, especially if the variety is high-dimensional.

\begin{example}
The projective plane is additive but not uniquely additive:
\begin{verbatim}
i1 : PP2 = toricProjectiveSpace 2;
i2 : additiveActions PP2
o2 = HashTable{completeCollection => {{-1, 1}, {0, 1}}}
               isAdditive => true
               isUniquelyAdditive => false
o2 : HashTable
\end{verbatim}

The projective variety corresponding to the del Pezzo polygon (see Definition~\ref{defi:dP-polytope}) does not admit an additive action:
\begin{verbatim}
i3 : V = transpose matrix {{1, 0}, {0, 1}, {-1, 0}, {0, -1}, {1, 1}, {-1, -1}};
              2        6
o3 : Matrix ZZ  <--- ZZ
i4 : P = convexHull V;
i5 : additiveActions P
o5 = HashTable{completeCollection => {}   }
               isAdditive => false
               isUniquelyAdditive => false
o5 : HashTable
\end{verbatim}
\end{example}

\subsection{listAdditiveSmoothFanoToricVarieties}\label{subsection:list-additive}\hfill

This method takes as input an integer $n~=~1,~\ldots,~6$. It returns a
list containing hash tables with keys \texttt{"databaseIndex"} (the index in the database of smooth Fano toric varieties included in the \texttt{NormalToricVarieties} package), \texttt{"isUniquelyAdditive"} (a \texttt{Boolean}, whether the variety is uniquely additive) and \texttt{"completeCollection"} (an object of type \texttt{List}, a complete collection of Demazure roots). This method is very fast, as the data is fetched from the pre-processed database included in our package.

\begin{example}
The following is the two-dimensional case:
\begin{verbatim}
i1 : listAdditiveSmoothFanoToricVarieties 2
o1 = {HashTable{completeCollection => {{-1, 1}, {0, 1}}},
                databaseIndex => 0                       
                isUniquelyAdditive => false              
     ------------------------------------------------------------------------
     HashTable{completeCollection => {{-1, 0}, {0, -1}}},
               databaseIndex => 1                        
               isUniquelyAdditive => true                
     ------------------------------------------------------------------------
     HashTable{completeCollection => {{-1, 0}, {0, 1}}},
               databaseIndex => 2                       
               isUniquelyAdditive => false              
     ------------------------------------------------------------------------
     HashTable{completeCollection => {{-1, 0}, {0, 1}}}}
               databaseIndex => 3
               isUniquelyAdditive => true
o1 : List
\end{verbatim}

The smooth Fano toric surface with database index 4 is not additive:
\begin{verbatim}
i2 : additiveActions(2, 4)
o2 = HashTable{completeCollection => {}   }
               isAdditive => false
               isUniquelyAdditive => false
o2 : HashTable
\end{verbatim}

There are 470 additive smooth Fano toric fivefolds:
\begin{verbatim}
i3 : #(listAdditiveSmoothFanoToricVarieties 5)
o3 = 470
\end{verbatim}
\end{example}

\subsection{listUniquelyAdditiveSmoothFanoToricVarieties}\label{subsection:list-uniquely-additive}\hfill

This method takes as input an integer $n~=~1,~\ldots,~6$. It returns a
list containing hash tables with keys \texttt{"databaseIndex"} (the index in the database of smooth Fano toric varieties included in the \texttt{NormalToricVarieties} package) and \texttt{"completeCollection"} (an object of type \texttt{List}, a complete collection of Demazure roots). This method is very fast, as the data is fetched from the pre-processed database included in our package.

\begin{example}
The following is the two-dimensional case:
\begin{verbatim}
i1 : listUniquelyAdditiveSmoothFanoToricVarieties 2
o1 = {HashTable{completeCollection => {{-1, 0}, {0, -1}}},
                databaseIndex => 1                        
     ------------------------------------------------------------------------
     HashTable{completeCollection => {{-1, 0}, {0, 1}}}}
               databaseIndex => 3
o1 : List
\end{verbatim}

There are 4 uniquely additive smooth Fano toric fivefolds:
\begin{verbatim}
i2 : #(listUniquelyAdditiveSmoothFanoToricVarieties 5)
o2 = 4
\end{verbatim}
\end{example}

\section{Applications to smooth Fano toric varieties}\label{section:applications}

In this section we apply our methods to determine existence and uniqueness of additive actions on smooth Fano toric varieties $X$ of dimension $\dim(X) \leq 6$. Most data was obtained from the GRDB \cite{gavin2009about}, who in turn obtained it from Øbro \cite{obro2007algorithm}. The case in which $\dim(X) = 2$ can be easily done by inspecting the corresponding polygons, but we include it for the sake of completeness. Existence in the $\dim(X) = 3$ case was previously done in \cite{HuanMont2020Fano}, who used more geometric techniques. Uniqueness in the $\dim(X) = 3$ case, along with existence and uniqueness in all higher-dimensional cases, are new results.

\subsection{Surfaces}\hfill

\begin{proposition}\label{prop:d-2-cohomology}
Let $P \subset N_{\QQ}$ be a smooth Fano polygon. Then:
\begin{enumerate}[1.]
    \item The only non-constant Betti number of $X_{P^{\vee}}$ is 
    \[
    h_1 = b_2 = \rho(X_{P^{\vee}}) = \# \mathcal{V}(P) - 2.
    \]
    \item $\# \mathcal{V}(P) = \# \mathcal{F}(P)$.
\end{enumerate}
\end{proposition}

\begin{proof}
\begin{enumerate}[1.]
    \item This follows from Proposition \ref{prop:poincare-duality} and from expanding the identity in Proposition \ref{prop:formula-betti} for $p = 1$.
    \item This follows from expanding the identity in Proposition \ref{prop:formula-betti} for $p = 0, 2$.
\end{enumerate}
\end{proof}

There are 5 isomorphism classes of smooth Fano toric surfaces, of which 1 is not additive, 2 are additive but not uniquely additive, and 2 are uniquely additive.

Table \ref{table:matching-debarre-d-2} contains all (uniquely) additive smooth Fano toric surfaces. These are matched to the data in the GRDB \cite{gavin2009about}. Non-additive surfaces are omitted. The column labels indicate, in order:

\begin{enumerate}[1.]
    \item \textbf{ID}: The ID in the GRDB.
    \item \textbf{Degree}: The degree $(-K_X)^2$.
    \item $\bm{b_2}$: The second Betti number, calculated using Proposition \ref{prop:d-2-cohomology}.
    \item $\bm{\#\mathcal{V}(P), \#\mathcal{F}(P)}$: The number of vertices and facets of the smooth Fano polygon $P \subset N_{\QQ}$, respectively.
    \item \textbf{Variety}: The corresponding smooth Fano toric surface $X_{P^{\vee}}$.
    \item \textbf{Uniquely additive?}: \emph{Yes} if the variety is uniquely additive, \emph{no} if it is not.
\end{enumerate}

\begin{longtable}[ht]{|p{1cm}|p{2cm}|p{1cm}|p{1.5cm}|p{1.5cm}|p{2cm}|p{2cm}|}
    \hline
    \textbf{ID} & \textbf{Degree} & $\bm{b_2}$ & $\bm{\#\mathcal{V}(P)}$ & $\bm{\#\mathcal{F}(P)}$ & \textbf{Variety} & \textbf{Uniquely additive?} \\ \hline
    1 & 7 & 3 & 5 & 5 & $\Bl_{p, q} \left ( \PP^2 \right)$ & Yes  \\ \hline
    3 & 8 & 2 & 4 & 4 & $\Bl_{p}(\mathbb{P}^2)$ & No  \\ \hline
    4 & 8 & 2 & 4 & 4 & $\mathbb{P}^1 \times \mathbb{P}^1$ & Yes  \\ \hline
    5 & 9 & 1 & 3 & 3 & $\PP^2$ & No  \\ \hline
    \caption{\label{table:matching-debarre-d-2} Additive smooth Fano toric surfaces.}
\end{longtable}

\subsection{Threefolds}\hfill

\begin{proposition}\label{prop:d-3-cohomology}
Let $P \subset N_{\QQ}$ be a smooth Fano polytope of dimension $3$. Then, the non-constant Betti numbers of $X_{P^{\vee}}$ are
\[
h_1 = b_2 = h_2 = b_4 = \rho(X_{P^{\vee}}) = \#\mathcal{V}(P) - 3.
\]
\end{proposition}

\begin{proof}
This follows from Propositions \ref{prop:poincare-duality} and \ref{prop:formula-betti}.
\end{proof}

There are 18 isomorphism classes of smooth Fano toric threefolds, of which 4 are not additive, 12 are additive but not uniquely additive, and 2 are uniquely additive.

Tables \ref{table:matching-mori-d-3} and \ref{table:matching-mori-d-3-uniquely-additive} contain all additive smooth Fano toric threefolds. These are matched to the data in the GRDB \cite{gavin2009about} and to the classifications in \cite{mori2003classification} and \cite{HuanMont2020Fano}. Non-additive threefolds are omitted. The column labels indicate, in order:

\begin{enumerate}[1.]
    \item \textbf{Nº}: The nº in \cite{mori2003classification}.
    \item \textbf{ID}: The ID in the GRDB.
    \item \textbf{Degree}: The degree $(-K_X)^3$, obtained from \cite{mori2003classification} and the GRDB.
    \item $\bm{b_2}$: The second Betti number, obtained from \cite{mori2003classification} and calculated using Proposition \ref{prop:d-3-cohomology}.
    \item $\bm{\#\mathcal{V}(P), \#\mathcal{F}(P)}$: The number of vertices and facets of the smooth Fano polytope $P \subset N_{\QQ}$, obtained from the GRDB.
    \item \textbf{Notation}: The notation in \cite{HuanMont2020Fano}.
    \item \textbf{Uniquely additive?}: \textit{Yes} if the variety is uniquely additive, \textit{no} if it is not.
    \item \textbf{Variety} (Table \ref{table:matching-mori-d-3-uniquely-additive}): The corresponding smooth Fano toric threefold $X_{P^{\vee}}$, obtained from \cite{mori2003classification}.  
\end{enumerate}

\begin{longtable}[ht]{|p{1cm}|p{1cm}|p{2cm}|p{1cm}|p{1.5cm}|p{1.5cm}|p{2cm}|p{2cm}|}
    \hline
    \textbf{Nº} & \textbf{ID} & \textbf{Degree} & $\bm{b_2}$ & $\bm{\#\mathcal{V}(P)}$ & $\bm{\#\mathcal{F}(P)}$ & \textbf{Notation} & \textbf{Uniquely additive?} \\ \hline
    - & 23 & 64 & 1 & 4 & 4 & $\mathbb{P}^3$ & No \\ \hline
    33 & 22 & 54 & 2 & 5 & 6 & $II_{33}$ & No \\ \hline
    34 & 19 & 54 & 2 & 5 & 6 & $II_{34}$ & No \\ \hline
    35 & 20 & 56 & 2 & 5 & 6 & $II_{35}$ & No \\ \hline
    36 & 7 & 62 & 2 & 5 & 6 & $II_{36}$ & No \\ \hline
    26 & 16 & 46 & 3 & 6 & 8 & $III_{26}$ & No \\ \hline
    27 & 21 & 48 & 3 & 6 & 8 & $III_{27}$ & Yes \\ \hline
    28 & 17 & 48 & 3 & 6 & 8 & $III_{28}$ & No \\ \hline
    29 & 6 & 50 & 3 & 6 & 8 & $III_{29}$ & No \\ \hline
    30 & 12 & 50 & 3 & 6 & 8 & $III_{30}$ & No \\ \hline
    31 & 11 & 52 & 3 & 6 & 8 & $III_{31}$ & No \\ \hline
    10 & 14 & 42 & 4 & 7 & 10 & $IV_{10}$ & Yes \\ \hline
    11 & 10 & 44 & 4 & 7 & 10 & $IV_{11}$ & No \\ \hline
    12 & 8 & 46 & 4 & 7 & 10 & $IV_{12}$ & No \\ \hline
    \caption{\label{table:matching-mori-d-3} Additive smooth Fano toric threefolds.}
\end{longtable}

\begin{table}[ht]
    \begin{tabular}{|p{1cm}|p{1cm}|p{2cm}|p{3cm}|}
        \hline
        \textbf{Nº} & \textbf{ID} & \textbf{Notation} & \textbf{Variety} \\ \hline
        27 & 21 & $III_{27}$ & $\mathbb{P}^1 \times \mathbb{P}^1 \times \mathbb{P}^1$ \\ \hline
        10 & 14 & $IV_{10}$ & $\mathbb{P}^1 \times \Bl_{p, q} \left ( \PP^2 \right)$ \\ \hline
    \end{tabular}
    \centering

\vspace{2mm}
    
    \caption{\label{table:matching-mori-d-3-uniquely-additive} Uniquely additive smooth Fano toric threefolds.}
\end{table}

\subsection{Fourfolds}\hfill

\begin{proposition}\label{prop:d-4-cohomology}
Let $P \subset N_{\QQ}$ be a smooth Fano polytope of dimension $4$. Then, the non-constant Betti numbers of $X_{P^{\vee}}$ are
\[
h_1 = b_2 = h_3 = b_6 = \rho(X_{P^{\vee}}) = \#\mathcal{V}(P) - 4
\]
and
\[
h_2 = b_4 = -2\#\mathcal{V}(P) + \#\mathcal{F}(P) + 6.
\]
\end{proposition}

\begin{proof}
Proposition \ref{prop:formula-betti} 2. implies that $h_1 = b_2 = f_2 - 2f_1 + 3\#\mathcal{V}(P) - 4 = h_3 = b_6 = \#\mathcal{V}(P) - 4$. Then,
\begin{equation*}
    f_2 - 2f_1 + 2\#\mathcal{V}(P) = 0. \tag{1}\label{eq:prop:d-4-cohomology-1}
\end{equation*}
On the other hand, Proposition \ref{prop:formula-betti} 1. implies that $h_0 = b_0 = \#\mathcal{F}(P) - f_2 + f_1 - \#\mathcal{V}(P) + 1 = h_4 = b_8 = 1$. Then,
\begin{equation*}
    \#\mathcal{F}(P) - f_2 + f_1 - \#\mathcal{V}(P) = 0. \tag{2}\label{eq:prop:d-4-cohomology-2}
\end{equation*}
Adding (\ref{eq:prop:d-4-cohomology-1}) and (\ref{eq:prop:d-4-cohomology-2}) gives $f_1 = \#\mathcal{V}(P) + \#\mathcal{F}(P)$, thus the identity in Proposition \ref{prop:formula-betti} implies that $h_2 = b_4 = f_1 - 3\#\mathcal{V}(P) + 6 = -2\#\mathcal{V}(P) + \#\mathcal{F}(P) + 6$.
\end{proof}

There are 124 isomorphism classes of smooth Fano toric fourfolds, of which 45 are not additive, 75 are additive but not uniquely additive, and 4 are uniquely additive.

Tables \ref{table:matching-batyrev-d-4} and \ref{table:matching-batyrev-d-4-uniquely-additive} contain all additive smooth Fano toric fourfolds. These are matched to the data in the GRDB \cite{gavin2009about} and to the classification in \cite{Baty1999Classification} (see also \cite{Sato2000Toward}). Non-additive fourfolds are omitted. The column labels indicate, in order:

\begin{enumerate}[1.]
    \item \textbf{Nº}: The nº in \cite{Baty1999Classification}.
    \item \textbf{ID}: The ID in the GRDB.
    \item \textbf{Degree}: The degree $(-K_X)^4$, obtained from \cite{Baty1999Classification} and the GRDB.
    \item $\bm{b_2, b_4}$: The second and fourth Betti numbers, obtained from \cite{Baty1999Classification} and calculated using Proposition \ref{prop:d-4-cohomology}.
    \item $\bm{\#\mathcal{F}(P), \#\mathcal{V}(P)}$: The number of facets and vertices of the smooth Fano polytope $P \subset N_{\QQ}$, obtained from the GRDB.
    \item \textbf{Notation}: The notation in \cite{Baty1999Classification}.
    \item \textbf{Uniquely additive?}: \textit{Yes} if the polytope is uniquely additive, \textit{no} if it is not.
    \item \textbf{Variety} (Table \ref{table:matching-batyrev-d-4-uniquely-additive}): The corresponding smooth Fano toric fourfold $X_{P^{\vee}}$, obtained from \cite{Baty1999Classification}.  
\end{enumerate}

\begin{longtable}[ht]{|p{1cm}|p{1cm}|p{2cm}|p{1cm}|p{1cm}|p{1.5cm}|p{1.5cm}|p{2cm}|p{2cm}|}
    \hline
    \textbf{Nº} & \textbf{ID} & \textbf{Degree} & $\bm{b_2}$ & $\bm{b_4}$ & $\bm{\#\mathcal{F}(P)}$ & $\bm{\#\mathcal{V}(P)}$ & \textbf{Notation} & \textbf{Uniquely additive?} \\ \hline
    1 & 147 & 625 & 1 & 1 & 5 & 5 & $\mathbb{P}^4$ & No \\ \hline
    2 & 25 & 800 & 2 & 2 & 8 & 6 & $B_1$ & No \\ \hline
    3 & 139 & 640 & 2 & 2 & 8 & 6 & $B_2$ & No \\ \hline
    4 & 144 & 544 & 2 & 2 & 8 & 6 & $B_3$ & No \\ \hline
    5 & 145 & 512 & 2 & 2 & 8 & 6 & $B_4$ & No \\ \hline
    6 & 138 & 512 & 2 & 2 & 8 & 6 & $B_5$ & No \\ \hline
    7 & 44 & 594 & 2 & 3 & 9 & 6 & $C_1$ & No \\ \hline
    8 & 141 & 513 & 2 & 3 & 9 & 6 & $C_2$ & No \\ \hline
    9 & 70 & 513 & 2 & 3 & 9 & 6 & $C_3$ & No \\ \hline
    10 & 146 & 486 & 2 & 3 & 9 & 6 & $C_4$ & No \\ \hline
    11 & 24 & 605 & 3 & 3 & 11 & 7 & $E_1$ & No \\ \hline
    12 & 128 & 489 & 3 & 3 & 11 & 7 & $E_2$ & No \\ \hline
    13 & 127 & 431 & 3 & 3 & 11 & 7 & $E_3$ & No \\ \hline
    14 & 30 & 592 & 3 & 4 & 12 & 7 & $D_1$ & No \\ \hline
    15 & 31 & 576 & 3 & 4 & 12 & 7 & $D_2$ & No \\ \hline
    16 & 49 & 560 & 3 & 4 & 12 & 7 & $D_3$ & No \\ \hline
    17 & 35 & 560 & 3 & 4 & 12 & 7 & $D_4$ & No \\ \hline
    18 & 42 & 496 & 3 & 4 & 12 & 7 & $D_5$ & No \\ \hline
    19 & 129 & 496 & 3 & 4 & 12 & 7 & $D_6$ & No \\ \hline
    20 & 97 & 486 & 3 & 4 & 12 & 7 & $D_7$ & No \\ \hline
    21 & 134 & 480 & 3 & 4 & 12 & 7 & $D_8$ & No \\ \hline
    22 & 66 & 464 & 3 & 4 & 12 & 7 & $D_9$ & No \\ \hline
    23 & 132 & 464 & 3 & 4 & 12 & 7 & $D_{10}$ & No \\ \hline
    24 & 117 & 459 & 3 & 4 & 12 & 7 & $D_{11}$ & No \\ \hline
    25 & 140 & 448 & 3 & 4 & 12 & 7 & $D_{12}$ & No \\ \hline
    26 & 143 & 432 & 3 & 4 & 12 & 7 & $D_{13}$ & No \\ \hline
    27 & 133 & 432 & 3 & 4 & 12 & 7 & $D_{14}$ & No \\ \hline
    28 & 135 & 432 & 3 & 4 & 12 & 7 & $D_{15}$ & No \\ \hline
    29 & 68 & 432 & 3 & 4 & 12 & 7 & $D_{16}$ & No \\ \hline
    33 & 41 & 529 & 3 & 5 & 13 & 7 & $G_1$ & No \\ \hline
    34 & 40 & 450 & 3 & 5 & 13 & 7 & $G_2$ & No \\ \hline
    35 & 64 & 433 & 3 & 5 & 13 & 7 & $G_3$ & No \\ \hline
    36 & 60 & 417 & 3 & 5 & 13 & 7 & $G_4$ & No \\ \hline
    37 & 69 & 406 & 3 & 5 & 13 & 7 & $G_5$ & No \\ \hline
    38 & 137 & 401 & 3 & 5 & 13 & 7 & $G_6$ & No \\ \hline
    39 & 26 & 558 & 4 & 5 & 15 & 8 & $H_1$ & No \\ \hline
    40 & 45 & 505 & 4 & 5 & 15 & 8 & $H_2$ & No \\ \hline
    41 & 28 & 478 & 4 & 5 & 15 & 8 & $H_3$ & No \\ \hline
    42 & 118 & 447 & 4 & 5 & 15 & 8 & $H_4$ & No \\ \hline
    43 & 123 & 415 & 4 & 5 & 15 & 8 & $H_5$ & No \\ \hline
    46 & 124 & 378 & 4 & 5 & 15 & 8 & $H_8$ & No \\ \hline
    49 & 74 & 480 & 4 & 6 & 16 & 8 & $L_1$ & No \\ \hline
    50 & 75 & 464 & 4 & 6 & 16 & 8 & $L_2$ & No \\ \hline
    51 & 83 & 448 & 4 & 6 & 16 & 8 & $L_3$ & No \\ \hline
    52 & 105 & 432 & 4 & 6 & 16 & 8 & $L_4$ & No \\ \hline
    53 & 95 & 416 & 4 & 6 & 16 & 8 & $L_5$ & No \\ \hline
    54 & 112 & 400 & 4 & 6 & 16 & 8 & $L_6$ & No \\ \hline
    55 & 106 & 384 & 4 & 6 & 16 & 8 & $L_7$ & No \\ \hline
    56 & 142 & 384 & 4 & 6 & 16 & 8 & $L_8$ & Yes \\ \hline
    57 & 130 & 384 & 4 & 6 & 16 & 8 & $L_9$ & No \\ \hline
    62 & 33 & 496 & 4 & 6 & 16 & 8 & $I_1$ & No \\ \hline
    63 & 29 & 463 & 4 & 6 & 16 & 8 & $I_2$ & No \\ \hline
    64 & 47 & 442 & 4 & 6 & 16 & 8 & $I_3$ & No \\ \hline
    65 & 38 & 433 & 4 & 6 & 16 & 8 & $I_4$ & No \\ \hline
    67 & 93 & 411 & 4 & 6 & 16 & 8 & $I_6$ & No \\ \hline
    68 & 37 & 400 & 4 & 6 & 16 & 8 & $I_7$ & No \\ \hline
    69 & 115 & 384 & 4 & 6 & 16 & 8 & $I_8$ & No \\ \hline
    70 & 94 & 390 & 4 & 6 & 16 & 8 & $I_9$ & No \\ \hline
    71 & 111 & 389 & 4 & 6 & 16 & 8 & $I_{10}$ & No \\ \hline
    72 & 59 & 384 & 4 & 6 & 16 & 8 & $I_{11}$ & No \\ \hline
    74 & 126 & 368 & 4 & 6 & 16 & 8 & $I_{13}$ & No \\ \hline
    77 & 61 & 385 & 4 & 7 & 17 & 8 & $M_1$ & No \\ \hline
    78 & 50 & 417 & 4 & 7 & 17 & 8 & $M_2$ & No \\ \hline
    79 & 58 & 369 & 4 & 7 & 17 & 8 & $M_3$ & No \\ \hline
    80 & 57 & 369 & 4 & 7 & 17 & 8 & $M_4$ & No \\ \hline
    81 & 110 & 364 & 4 & 7 & 17 & 8 & $M_5$ & No \\ \hline
    84 & 71 & 442 & 5 & 8 & 20 & 9 & $Q_1$ & No \\ \hline
    85 & 79 & 405 & 5 & 8 & 20 & 9 & $Q_2$ & No \\ \hline
    86 & 73 & 394 & 5 & 8 & 20 & 9 & $Q_3$ & No \\ \hline
    87 & 77 & 405 & 5 & 8 & 20 & 9 & $Q_4$ & No \\ \hline
    88 & 81 & 373 & 5 & 8 & 20 & 9 & $Q_5$ & No \\ \hline
    89 & 84 & 368 & 5 & 8 & 20 & 9 & $Q_6$ & No \\ \hline
    90 & 91 & 363 & 5 & 8 & 20 & 9 & $Q_7$ & No \\ \hline
    91 & 90 & 352 & 5 & 8 & 20 & 9 & $Q_8$ & No \\ \hline
    93 & 102 & 336 & 5 & 8 & 20 & 9 & $Q_{10}$ & No \\ \hline
    94 & 120 & 336 & 5 & 8 & 20 & 9 & $Q_{11}$ & Yes \\ \hline
    105 & 89 & 332 & 5 & 9 & 21 & 9 & $R_1$ & No \\ \hline
    117 & 62 & 307 & 5 & 11 & 23 & 9 & See Remark \ref{remark:special-polytope} & Yes \\ \hline
    119 & 98 & 294 & 6 & 11 & 25 & 10 & $\Bl_{p, q} \left ( \PP^2 \right) \times \Bl_{p, q} \left ( \PP^2 \right)$ & Yes \\ \hline
    \caption{\label{table:matching-batyrev-d-4} Additive smooth Fano toric fourfolds.}
\end{longtable}

\begin{table}[ht]
    \begin{tabular}{|p{1cm}|p{1cm}|p{4cm}|p{4cm}|}
        \hline
        \textbf{Nº} & \textbf{ID} & \textbf{Notation} & \textbf{Variety} \\ \hline
        56 & 142 & $L_8$ & $\mathbb{P}^1 \times \mathbb{P}^1 \times \mathbb{P}^1 \times \mathbb{P}^1$ \\
        \hline
        94 & 120 & $Q_{11}$ & $\mathbb{P}^1 \times \mathbb{P}^1 \times \Bl_{p, q} \left ( \PP^2 \right)$ \\
        \hline
        117 & 62 & See Remark \ref{remark:special-polytope} & See Remark \ref{remark:special-polytope} \\
        \hline
        119 & 98 & $\Bl_{p, q} \left ( \PP^2 \right) \times \Bl_{p, q} \left ( \PP^2 \right)$ & $\Bl_{p, q} \left ( \PP^2 \right) \times \Bl_{p, q} \left ( \PP^2 \right)$ \\
        \hline
    \end{tabular}
    \centering

\vspace{2mm}
    
    \caption{\label{table:matching-batyrev-d-4-uniquely-additive} Uniquely additive smooth Fano toric fourfolds.}
\end{table}

\begin{remark}\label{remark:special-polytope}
    There is a small misprint in Batyrev's classification. Namely, the types of polytopes Nº 117 and 118 are exchanged. The correct statement is that polytope Nº 117 is isomorphic to the convex hull of the subset
    \[
    \left \{ \pm e_1, \ldots, \pm e_n, -\sum_{i = 1}^n e_i \right \} \subset N_{\QQ}.
    \]
    This can be seen, for example, by using Proposition \ref{prop:d-4-cohomology} to calculate the corresponding invariants. The corresponding fourfold is isomorphic to $\Bl_{p} \left (\left ( \PP^1 \right )^4 \right)$, the blow-up of $\left ( \PP^1 \right )^4$ at one torus-invariant point.
\end{remark}

\subsection{Fivefolds and sixfolds}\hfill

Data for fivefolds and sixfolds can be accessed through the \texttt{listAdditiveSmoothFanoToricVarieties} and \texttt{listUniquelyAdditiveSmoothFanoToricVarieties} methods of the \emph{Macaulay2} package, see Subsections \ref{subsection:list-additive} and \ref{subsection:list-uniquely-additive}.

There are $866$ isomorphism classes of smooth Fano toric fivefolds, of which $396$ are not additive, $466$ are additive but not uniquely additive, and $4$ are uniquely additive.

There are $7622$ isomorphism classes of smooth Fano toric sixfolds, of which $4194$ are not additive, $3420$ are additive but not uniquely additive, and $8$ are uniquely additive.

\subsection{Uniquely additive varieties}\hfill

\begin{definition}
    Let $P, Q \subset N_{\QQ}$ be two polytopes. The \textbf{free sum} $P \oplus Q \subset N_{\QQ} \times N_{\QQ}$ is the polytope defined as the convex hull of the subset
    \[
    (P \times \{0_{N_{\QQ}}\}) \cup (\{0_{N_{\QQ}}\} \times Q) \subset N_{\QQ} \times N_{\QQ}.
    \]
\end{definition}

\begin{proposition}\cite{CoxLittSche2011Toric}
    Let $P, Q \subset N_{\QQ}$ be two polytopes such that $0_{N_{\QQ}} \in \RelInt(P), \RelInt(Q)$. Then:
    \begin{enumerate}[1.]
        \item $(P \oplus Q)^{\vee} = P^{\vee} \times Q^{\vee}$.
        \item The product $P^{\vee} \times Q^{\vee}$ corresponds to the toric variety $X_{P^{\vee}} \times X_{Q^{\vee}}$ with the Segre embedding.
    \end{enumerate}
\end{proposition}

\begin{definition}\label{defi:dP-polytope}
    Let $n \in \ZZ^+$ be even.
    \begin{enumerate}[1.]
        \item The \textbf{del Pezzo polytope} is the polytope defined as the convex hull of the subset
        \[
        \left \{ \pm e_1, \ldots, \pm e_n, \pm \sum_{i = 1}^n e_i \right \} \subset N_{\QQ}.
        \]
        \item The \textbf{pseudo-del Pezzo polytope} is the polytope defined as the convex hull of the subset
        \[
        \left \{ \pm e_1, \ldots, \pm e_n, -\sum_{i = 1}^n e_i \right \} \subset N_{\QQ}.
        \]
    \end{enumerate}
\end{definition}

\begin{remark}
    The toric variety corresponding to the del Pezzo polygon (see Definition~\ref{defi:dP-polytope}) is isomorphic to both $\Bl_{p} \left ( \left ( \PP^1 \right )^2 \right)$ and $\Bl_{p, q} \left ( \PP^2 \right )$.
\end{remark}

\begin{proposition}\cite{casagrande2003centrally}
    Let $P \subset N_{\QQ}$ be a smooth Fano polytope with $n$ distinct unordered pairs $(v_1, v_2)$ of vertices such that $v_1 = -v_2$. Then, $P$ is isomorphic to a free sum of intervals $[-1, 1]$, del Pezzo polytopes and pseudo-del Pezzo polytopes.
\end{proposition}

\begin{proposition}
    Let $P \subset N_{\QQ}$ be a polytope.
    \begin{enumerate}[1.]
        \item If $P$ is isomorphic to a del Pezzo polytope, then the corresponding toric variety $X_{P^{\vee}}$ is not additive.
        \item If $P$ is isomorphic to a pseudo-del Pezzo polytope, then the corresponding toric variety $X_{P^{\vee}}$ is uniquely additive.
    \end{enumerate}
\end{proposition}

\begin{proof}
    \begin{enumerate}[1.]
        \item The dual polytope is
        \[
        P^{\vee} = \{(x_1, \ldots, x_n) \in M_{\QQ}: \quad |x_i| \leq 1 \text{ for all } i = 1, \ldots, n \text{ and } |x_1 + \cdots + x_n| \leq 1\}.
        \]
        It is easy to see that all facets of $P^{\vee}$ have empty relative interior, so the set of Demazure roots of $P^{\vee}$ is empty (see Remark \ref{remark:smooth-fano-toric-varieties}).
        \item The dual polytope is
        \[
        P^{\vee} = \{(x_1, \ldots, x_n) \in M_{\QQ}: \quad |x_i| \leq 1 \text{ for all } i = 1, \ldots, n \text{ and } -x_1 - \cdots - x_n \leq 1\}.
        \]
        There exist rays $\rho_1, \ldots, \rho_n \in \Sigma_{P^{\vee}}(1)$ such that $B = \{p_1, \ldots, p_n\} = \{e_1, \ldots, e_n\}$ is a basis of the lattice $N$ (here, $\{e_1, \ldots, e_n\}$ is the canonical basis of $N$), the remaining rays in $\Sigma(1)$ are contained in the negative octant of $B$, and $\mathfrak{R}_i = \{-p_i^{\vee}\} = \{-e_i^{\vee}\}$ for all $i = 1, \ldots, n$. The result then follows by Theorem \ref{theorem:uniqueness}.
    \end{enumerate}
\end{proof}

The following theorem is a consequence of applying our methods to the data in the GRDB \cite{gavin2009about}.

\begin{theorem}\label{theorem:uniquely-additive-smooth-fano}
    Let $X = X_{P^{\vee}}$ be a smooth Fano toric variety such that $\dim(X) \leq 6$. The following are equivalent:
    \begin{enumerate}[1.]
        \item $X$ is uniquely additive.
        \item $X$ is isomorphic to the product of projective lines $\PP^1$, blow-ups of $\left ( \PP^1 \right )^n$ with $n$ even at one torus-invariant point $\Bl_p \left ( \left ( \PP^1 \right )^n \right)$ and, if $\dim(X_P) = 6$, $X_{{\mathfrak{P}}^{\vee}}$, where $\mathfrak{P} \subset N_{\QQ}$ is the polytope with \emph{ID} $5816$ in the \emph{GRDB}.
        \item $P \subset N_{\QQ}$ is isomorphic to a free sum of intervals $[-1, 1]$, pseudo-del Pezzo polytopes and, if $\dim(P) = 6$, the polytope $\mathfrak{P}$ with \emph{ID} $5816$ on the \emph{GRDB}.
    \end{enumerate}
    In particular, the sequence $(c_n)_{n \geq 1}$, where $c_n$ is the number of isomorphism classes of uniquely additive smooth Fano toric varieties of dimension $n$, starts with $1, 2, 2, 4, 4, 8, \ldots$
\end{theorem}

\section{Further directions}\label{section:further-directions}

We welcome any contributions and feature requests for the \emph{Macaulay2} package.

Since the enumeration of classes uniquely additive smooth Fano toric varieties is very simple up to $n = 6$ (see Theorem \ref{theorem:uniquely-additive-smooth-fano}), it would be interesting to characterise and enumerate these classes for higher or arbitrary values of $n$.

More generally, it would also be interesting to enumerate classes of additive complete toric varieties. In order to get finite values, one may consider fans whose primitive ray generators have bounded coordinates. To this effect, Proposition \ref{proposition:negative-octant-criterion} may be a good starting point. For example:

\begin{proposition}\label{proposition:oeis}
    Let $(d_n)_{n \geq 1}$ be the sequence counting, up to the action of $\GL(N)$, the number of 1-skeletons $\Sigma(1)$ of complete toric varieties $X_{\Sigma}$ of dimension $n$ such that:
    \begin{enumerate}[1.]
        \item All primitive lattice vectors $p_1, \ldots, p_{\# \Sigma(1)}$ on the rays $\rho_1, \ldots, \rho_{\# \Sigma(1)} \in \Sigma(1)$ are contained in $[-1, 1]^n$.
        \item For all $i = 1, \ldots, n$, $p_i = e_i$, the $i$-th standard basis vector of $N_{\mathbb{Q}}$.
        \item For all $i = n + 1, \ldots, \# \Sigma(1)$, $p_i$ is in the negative octant of the standard basis of $N_{\QQ}$. 
    \end{enumerate}
    Then, all such varieties $X_{\Sigma}$ are additive, and, for all $n \in \ZZ_+$,
    \[
    1 = d_1 < \cdots < d_n \leq \OEIShypergraphs(n),
    \]
    where $\OEIShypergraphs$ is the \emph{OEIS} \cite[A055621]{oeis} sequence counting the number of hypergraphs on $n$ unlabelled vertices with no isolated vertices. 
\end{proposition}

\begin{proof}
    The first statement is a consequence of Proposition \ref{proposition:negative-octant-criterion}. We now prove the second statement.
    
    For all $i = n + 1, \ldots, \# \Sigma(1)$, we map $p_i$ to the set given by its support. For example, $(0, -1, 0, -1)$ is mapped to $\{2, 4\}$. The image of $\{p_i\}_{i = n + 1, \ldots, \# \Sigma(1)}$ is a hypergraph $G = (V, E)$ on the vertices $V = \{1, \ldots, n\}$ with no isolated vertices (since $\Sigma$ is complete). The symmetric group $\mathfrak{S}_n$ acts on fans $\Sigma$ in $N_{\QQ}$ by permuting coordinates and on hypergraphs $G = (\{1, \ldots, n\}, E)$ by permuting vertices. Since permutation matrices have determinant $\pm 1$, all fans $\Sigma$ in $N_{\QQ}$ in a given orbit belong to the same isomorphism class. On the other hand, the orbit of a hypergraph $G = (\{1, \ldots, n\}, E)$ is the underlying hypergraph $G'$ on $n$ unlabelled vertices.
\end{proof}

Proposition \ref{proposition:oeis} is inspired by the following conjecture of Ewald, although it is \emph{not} true that all additive smooth Fano toric varieties satisfy conditions 1. to 3. -- counterexamples may be found computationally already for $n \leq 4$. This is because, even if their corresponding smooth Fano polytopes satisfy Ewald's conjecture, the change of basis mapping the vertices to the standard basis of $N_{\mathbb{Q}}$ and vectors in the negative octant may \emph{stretch} the 1-skeletons.

\begin{conjecture}[Ewald]
Let $P \subset N_{\QQ}$ be a smooth Fano polytope of dimension $n$. Then, there exists $P' \subset [-1, 1]^n \subset N_{\QQ}$ such that $P \cong P'$.
\end{conjecture}

\section{Acknowledgements}

Some of the results and tables in this paper first appeared in FL's professional degree thesis at Universidad T\'ecnica Federico Santa Mar\'ia.

We would like to thank Alexander Kasprzyk for kindly offering to send us the relevant data in the GRDB. FL would also like to thank Crist\'obal Montecino for valuable discussions on computational aspects of toric varieties.

PM was partially supported by Fondecyt ANID Projects 1231214 and 1240101.

\bibliographystyle{amsalpha}
\bibliography{main}

\end{document}